\documentclass{patmorin}
\usepackage{pat}
\usepackage[T1]{fontenc}
\usepackage[utf8]{inputenc}
\usepackage{todonotes,float,subcaption}
\usepackage{algorithm,algpseudocode}
\usepackage{indentfirst}

\usepackage[mathlines]{lineno}
\newcommand*\patchAmsMathEnvironmentForLineno[1]{%
 \expandafter\let\csname old#1\expandafter\endcsname\csname #1\endcsname
 \expandafter\let\csname oldend#1\expandafter\endcsname\csname end#1\endcsname
 \renewenvironment{#1}%
    {\linenomath\csname old#1\endcsname}%
    {\csname oldend#1\endcsname\endlinenomath}}%
\newcommand*\patchBothAmsMathEnvironmentsForLineno[1]{%
 \patchAmsMathEnvironmentForLineno{#1}%
 \patchAmsMathEnvironmentForLineno{#1*}}%
\AtBeginDocument{%
\patchBothAmsMathEnvironmentsForLineno{equation}%
\patchBothAmsMathEnvironmentsForLineno{align}%
\patchBothAmsMathEnvironmentsForLineno{flalign}%
\patchBothAmsMathEnvironmentsForLineno{alignat}%
\patchBothAmsMathEnvironmentsForLineno{gather}%
\patchBothAmsMathEnvironmentsForLineno{multline}%
}

\usepackage[longnamesfirst,numbers,sort&compress]{natbib}

\definecolor{brightmaroon}{rgb}{0.76, 0.13, 0.28}
\definecolor{linkblue}{rgb}{0, 0.337, 0.227}
\newcommand{\defin}[1]{\emph{\textcolor{brightmaroon}{#1}}}
\makeatletter
\def\mathcolor#1#{\@mathcolor{#1}}
\def\@mathcolor#1#2#3{%
  \protect\leavevmode
  \begingroup
    \color#1{#2}#3%
  \endgroup
}
\makeatother

\newcommand{\cay}{\mathsf{Cay}}

\usepackage[inline]{enumitem}
\newcommand{\sm}{\smallsetminus}
\newcommand{\stww}{\mathsf{stww}}
\newcommand{\tww}{\mathsf{tww}}

\title{\MakeUppercase{Accessibility and twin-width}}

\author{
George Kontogeorgiou\thanks{Center for Mathematical Modeling (CNRS IRL2807), University of Chile, Santiago, Chile. \\Supported by ANID Basal Grant CMM FB210005 and ANID-FONDECYT Postdoctorado Grant No. 3250479. \\Email:{ \tt gkontogeorgiou@dim.uchile.cl}
}\hspace{5mm}Bobby Miraftab\thanks{School of Computer Science, Carleton University, Ottawa, Canada. Email: {\tt bobby.miraftab@gmail.com}}}

\begin{document}
\maketitle

\begin{abstract}
We show that finite twin-width does not imply accessibility for finitely generated groups, which answers a question of Esperet. That is, we prove that there exists a finitely generated group $\Gamma$ that has finite uniform twin-width but is not accessible. In particular, for every finite generating set $S$ of $\Gamma$, the Cayley graph
$\cay(\Gamma;S)$ has finite twin-width but is not accessible.
The example is obtained by combining Wilkes's construction of a finitely
generated inaccessible residually $2$-finite group \cite{Wilkes2019} with a result of Bonnet, Geniet, Tessera and Thomass\'e regarding the twin-width of groups acting faithfully on regular trees \cite{bonnet2022twin3}.
\end{abstract}

\section{Introduction}

\emph{Twin-width} is a graph parameter introduced by Bonnet, Kim, Thomassé and Watrigant in Twin-width I \cite{bonnet2021twin}. It measures the complexity of a graph through contraction sequences.
Bounded twin-width has proved to be a robust tameness condition, with connections to structural graph theory, finite model theory, and algorithmic meta-theorems; see also the rest of the Twin-width series \cite{bonnet2021twin2,bonnet2024twin,bonnet2024twin2,bonnet2022twin,bonnet2022twin2,bonnet2022twin3,bonnet2022twin4}.

Twin-width has also found a place in coarse graph theory and geometric group theory, where it can be calculated over infinite graphs, particularly Cayley graphs, as the supremum of the twin-widths of all finite induced subgraphs. In Twin-width II \cite{bonnet2021twin2} it was proved that finiteness of twin-width is a group invariant among finitely generated groups. In fact, in the class of graphs of bounded degree, twin-width is a quasi-isometric invariant, as seen in Twin-width VII \cite{bonnet2022twin3}. 

It was somewhat more challenging to prove that finiteness of twin-width is a meaningful group invariant, that is, that there exist finitely generated groups with infinite twin-width. This was achieved in \cite{bonnet2022twin3} through the following approach.

A class of labeled graphs is called \emph{small} if for each $n\in\mathbb{N}$ it contains at most $2^{O(n)}\cdot n!$ graphs on $n$ vertices. It was shown in \cite{bonnet2021twin2} that every class of bounded twin-width is small. Some years earlier, Osajda \cite{osajda2014small} had used small cancellation theory to prove that, given a class of graphs $\mathcal{C}$ that obeys certain conditions with respect to maximum degree, girth and diameter, it is possible to construct a (locally finite) Cayley graph $G$ such that every graph in $\mathcal{C}$ is an induced subgraph of $G$. The contribution of Bonnet, Geniet, Tessera and Thomass\'e \cite{bonnet2022twin3} to this question was to construct a specific class $\mathcal{C}_1$ of graphs that obeys Osajda's conditions and is not small. In particular, the twin-width of the members of $\mathcal{C}_1$ is unbounded, hence the twin-width of the Cayley graph $G$ that contains all the members of $\mathcal{C}_1$ as induced subgraphs is infinite.

Although this approach resolves the question of the existence of groups with infinite twin-width, it is not entirely satisfactory, due to being somewhat artificial. 
It would be pleasant to know that there is some significant property of finitely generated groups that follows from finite twin-width. 
Accessibility appears to be a good candidate for this. 
Let us briefly recall its graph-theoretic definition: a locally finite graph $G$ is \emph{accessible} if there exists a constant $k\in\mathbb{N}$ such that any two ends of $G$ can be separated by at most $k$ edges. 
A finitely generated group is accessible if one (equivalently, all) of its Cayley graphs is (are) accessible. There are also other equivalent definitions for accessibility from the graph-theoretic perspective, see \cite{HamannMiraftab2020,HamannLehnerMiraftabRuehmann2022}. 

In personal communication, Louis Esperet asked whether finite twin-width implies accessibility for groups. 
Since bounded twin-width is a strong structural condition in many finite settings, one might hope that a locally finite graph of finite twin-width
will necessarily be accessible. 
The purpose of this note is to show that this is not the
case. Our result is the following.

\begin{thm}
\label{thm:main}
There exists a finitely generated group $\Gamma$ of finite twin-width that is not accessible.
\end{thm}

The proof is short and relies on two results from the literature. 

The group-theoretic input is a theorem of Wilkes.
For every prime $p$, there exists a finitely generated inaccessible residually $p$-finite discrete group. 
We use this result for $p=2$. Let $\Gamma$ be such a group.  
Since $\Gamma$ is residually $2$-finite, one can choose a descending chain of normal subgroups
\[
        \Gamma=N_0\triangleright N_1\triangleright N_2\triangleright \cdots
\]
with $[N_i:N_{i+1}]=2$ for every $i$ and
\[
        \bigcap_i N_i=\{1\}.
\]
The associated coset tree has level $i$ consisting of the right cosets
$N_i\gamma$, and each vertex has exactly $2$ children.  Hence this tree is the
rooted binary tree $T_2$.  
Right multiplication gives an action of $\Gamma$ on this tree, and the triviality of $\bigcap_i N_i$ makes the action faithful.
Thus $\Gamma \leq \operatorname{Aut}(T_2)$.

The twin-width input comes from \cite{bonnet2022twin3}. The authors develop a new version of twin-width for groups, called \defin{uniform twin-width}, that is defined through permutation actions. 
Uniform twin-width is preserved by many natural group operations. 
Moreover, finiteness of uniform twin-width implies finiteness of twin-width. The theorems that we need from \cite{bonnet2022twin3} are the following: that the automorphism group of the rooted binary tree has uniform twin-width $2$, and that uniform twin-width is monotone under taking subgroups. Applying these two results to the faithful action $\Gamma\curvearrowright T_2$ shows that the inaccessible group $\Gamma$ has finite twin-width as a group.  

The example shows that finite twin-width and accessibility capture different forms of tameness. Finite twin-width controls the behavior of finite induced subgraphs, or equivalently finite pieces of the Cayley graph, through contraction sequences. Accessibility, by contrast, controls the global structure of ends. The theorem demonstrates that bounded finite-piece complexity does not force a
uniform bound on separating ends, even in the highly homogeneous setting of Cayley graphs.

\section{Definitions and Preliminaries}

\subsection{Graphs \& Groups}

A graph is \defin{locally finite} if every vertex has finite degree. Throughout the paper, all infinite graphs are countably infinite and locally finite. A \defin{ray} is an infinite one-way path. Two rays are \defin{equivalent} if they are joined by an infinite number of pairwise disjoint paths, and an \defin{end} is an equivalence class of rays.

For a group $\Gamma$ and a \textbf{finite} generating set $S$ of $\Gamma$, we define the \defin{Cayley graph} \[\cay(\Gamma;S):=(\Gamma, \{\{g,gs\}|g\in\Gamma, s\in S\}).\]  

Given two metric spaces $X$ and $Y$ (for example, two graphs with the usual graph metric), a \defin{quasi-isometry} is a map $f\colon X\rightarrow Y$ with the following properties:
\begin{itemize}
    \item there exists $C>0$ such that $C^{-1}\cdot d(f(x_1),f(x_2))-C\leq d(x_1,x_2)\leq C\cdot d(f(x_1),f(x_2))+C$ for every $x_1, x_2\in X$; 
    \item there exists $\delta >0$ such that for every $y\in Y$ there is $x\in X$ with $d(f(x),y)<\delta$.
\end{itemize}
Different Cayley graphs of a finitely generated group are quasi-isometric.  

A graph property $P$ is called a \defin{group invariant} if, for every group $\Gamma$, either all Cayley graphs of $\Gamma$ are in $P$ or none of them. A graph property is a \defin{quasi-isometric invariant} if it is preserved by quasi-isometries. Naturally, every quasi-isometric invariant is a group invariant. We note that the number of ends of a graph is a quasi-isometric invariant, so it is well-defined to talk about the number of ends of a finitely generated group. A classic theorem of Freudenthal \cite{freudenthal1944enden} and Hopf \cite{hopf1943enden} says that every group has zero, one, two, or infinite ends.

Finally, a group $\Gamma$ is \defin{residually $p$-finite} if for each $g\in \Gamma\sm  \{1\}$ there exists $N\triangleleft\Gamma$ such that $\Gamma/N$ is a finite $p$-group and $g\notin N$. A common equivalent definition is to have \[\bigcap_{N\triangleleft \Gamma, \Gamma/N \text{ finite } p-\text{group}}N=\{1\}.\] It is a simple fact, proved here for completeness, that: 

\begin{lem}\label{chain}

Given a countably infinite residually $p$-finite group $\Gamma$, there exists a descending chain of normal subgroups
\[
        \Gamma=N_0\triangleright N_1\triangleright N_2\triangleright \cdots
\]
with $[N_i:N_{i+1}]=p$ for every $i$ and
\[
        \bigcap_i N_i=\{1\}.
\]

\end{lem}

\begin{proof}
    Enumerate the nontrivial elements of $\Gamma$ as $\Gamma\sm \{1\}=\{g_1,g_2,\ldots\}$.
Since $\Gamma$ is residually $p$-finite, for each $i$ there is a normal subgroup
$M_i\triangleleft \Gamma$ of finite $p$-power index such that
\[
        g_i\notin M_i.
\]
We set $L_n=M_1\cap \cdots \cap M_n$.
Then each $L_n$ is normal and has finite $p$-power index in $\Gamma$, and
\[
        \bigcap_{n\geq 1} L_n=\{1\}.
\]

We refine the descending chain $\Gamma=L_0\triangleright L_1\triangleright L_2\triangleright \cdots$ so that every step has index $p$. Indeed, for each $n$, the quotient
$L_n/L_{n+1}$ is a finite $p$-group, hence has a composition series whose
successive factors have order $p$. Pulling this series back to $L_n$ gives
a finite chain from $L_n$ to $L_{n+1}$ with successive indices $p$.
Concatenating these finite refinements gives a descending chain
\[
        \Gamma=N_0\triangleright N_1\triangleright N_2\triangleright \cdots
\]
such that
\[
        [N_i:N_{i+1}]=p
        \qquad\text{for every }i,
        \qquad\text{and}\qquad
        \bigcap_{i\geq 0}N_i=\{1\}.\qedhere
\]
\end{proof}

\subsection{Accessibility}

Stallings' Theorem \cite{Stallings1971} is a fundamental result of infinite group theory. 
It states that every group with at least two ends \defin{splits} over a finite subgroup. 
That is, if $\Gamma$ is such a group, then either it can be written as a free product with amalgamation over a finite subgroup, or as an HNN extension with respect to an isomorphism between two finite subgroups. 
For more details, see \cite{dicks1989groups} for a classic treatment, or \cite{serre2002trees} from the perspective of Bass-Serre theory. 

In simple terms, Stallings' Theorem provides a canonical way to decompose a group with at least two ends, called a \defin{JSJ-splitting}. For certain groups, iterative JSJ-splittings of their factors halt after a finite number of steps. 
Such groups are termed \defin{accessible}. 
They are the groups that can conveniently be described in terms of JSJ-splittings yielding a finite number of finite and 1-ended factors. Dunwoody \cite{dunwoody1985accessibility} proved that every finitely presented group is accessible. He also gave the first example of a finitely generated group that is not accessible, see \cite{Dunwoody1993}.
In the context of finitely generated groups, Thomassen and Woess \cite{ThomassenWoess1993} proved a very intuitive equivalent definition from the perspective of graph theory: a group is accessible if in one (equivalently, in each) of its Cayley graphs each pair of ends is separated by a vertex-cut (equivalently, edge-cut) of bounded size. The definition of Thomassen and Woess was transformative for infinite graph theory, and allowed graph theorists to study accessibility and Stallings decompositions in more general contexts, such as in quasi-transitive graphs \cite{hamann2022stallings}. 

In 2019, Wilkes \cite{Wilkes2019} constructed finitely generated residually $p$-finite groups that are not accessible. It is precisely among these groups that we will obtain a witness for ~\Cref{thm:main}. 

\begin{lem}
{\rm \cite[Theorem 5.1]{Wilkes2019}}
\label{lem:wilkes}
For every prime $p$, there exists a finitely generated, inaccessible, residually $p$-finite discrete group.
\end{lem}

\subsection{Twin-width}

A \defin{vertex identification} on a graph $G$ is an operation in which two distinct vertices $u,v\in V(G)$ are deleted from $G$, and a new vertex $w$ is added so that $N(w)=N(u)\cup N(v)\sm \{u,v\}$. 
Suppose that $G$ has $n\in\mathbb{N}$ vertices and let $G=:G_n\rightarrow\cdots\rightarrow G_1$ be a sequence of vertex identifications. 
The \defin{width} of this sequence is defined to be $\max\{\Delta(G_n),\cdots,\Delta(G_1)\}$. 
The minimum width among all sequences of vertex identifications of length $n$ applied to $G$ is called the \defin{strict twin-width} of $G$ and is denoted \defin{$\stww(G)$}. 
If $G$ has infinitely many vertices, then its strict twin-width is defined to be the supremum of the strict twin-widths of all its induced subgraphs. 

Strict twin-width is a variant of twin-width introduced in \cite{bonnet2022twin3}. It relates to ordinary twin-width, denoted \defin{$\tww(G)$}, as follows. 

\begin{lem}
{\rm \cite[Equation (1), Section 3.1.1]{bonnet2022twin3}}
\label{lem:strict-ordinary}
For every finite graph $G$,
\[
        \max\{\Delta(G),\operatorname{tww}(G)\}
        \leq
        \operatorname{\stww}(G)
        \leq
        \Delta(G)+\operatorname{tww}(G).
\]
\end{lem}
In particular, finiteness of twin-width and of strict twin-width are equivalent properties for graphs of bounded degree. 
Because of this, it is convenient to work only with strict twin-width in this paper, and ever refer to it merely as ``twin-width".

Although finiteness of twin-width is a group invariant, the specific value of the parameter is not, as it depends on the particular Cayley graph. 
In order to define twin-width for groups, Bonnet, Geniet, Tessera and Thomass\'e \cite{bonnet2022twin3} introduced \defin{uniform twin-width} and showed that it is stable under a variety of standard group operations. 
For a group $\Gamma$, we denote its uniform twin-width by \defin{$\mathsf{utww}(\Gamma)$}. 
It is always equal to a no-negative integer or infinity. 
We remark that if a group has finite uniform twin-width, then all of its Cayley graphs have finite twin-width. 

Regarding uniform twin-width, we need the following two results.

\begin{lem}{\rm \cite[Lemma 6.1]{bonnet2022twin3}}\label{lem:subgroup}
    Uniform twin-width is monotone under taking subgroups. That is, if $H\leq\Gamma$, then $utww(H)\leq utww(\Gamma)$.  
\end{lem}

\begin{lem}
{\rm \cite[Proposition 6.18]{bonnet2022twin3}}
\label{lem:aut-binary-tree}
Let $T_2$ be the infinite rooted binary tree. 
The group $\operatorname{Aut}(T_2)$ has uniform twin-width $2$.
\end{lem}

\section{Finite twin-width does not imply accessibility for Cayley graphs}


In this section, we prove~\Cref{thm:main}.

\begin{lem}
\label{lem:residually-p-tree}
Let $\Gamma$ be a countably infinite residually $p$-finite group. Then $\Gamma$ admits
a faithful action on the rooted $p$-regular tree. 
In particular, if $p=2$, then $\Gamma \leq \operatorname{Aut}(T_2)$.

\end{lem}

\begin{proof}

Consider the descending chain of normal subgroups guaranteed by \Cref{chain}. Now define a rooted tree $T_p$ as follows. 
The vertices at level $i$ are the right cosets $N_i\gamma$, for $\gamma\in\Gamma$.
The parent of $N_{i+1}\gamma$ is $N_i\gamma$. 
Since $[N_i:N_{i+1}]=p$, each vertex has exactly $p$ children. Thus $T_p$ is the rooted $p$-ary tree.

The group $\Gamma$ acts on $T_p$ by right multiplication: $(N_i\gamma)\cdot h=N_i\gamma h$.
This action preserves levels and parent-child relations, so it is an action by
root-preserving automorphisms of $T_p$.
The action is faithful. 
Indeed, if $h\in\Gamma$ acts trivially on $T_p$, then it fixes the
vertex $N_i$ at every level. 
Hence we have $N_i h=N_i$ for every $i$, so $h\in N_i$ for every $i$. 
Therefore we obtain $$h\in \bigcap_i N_i=\{1\},$$ which yields a contradiction.
\end{proof}

\begin{thm}
\label{thm:cayley-finite-tww-not-accessible}
There exists a finitely generated group $\Gamma$ that has finite uniform twin-width but is not accessible. 
In particular, for every finite generating set $S$, the Cayley graph
$\cay(\Gamma;S)$ has finite twin-width but is not accessible.
\end{thm}

\begin{proof}
Apply \Cref{lem:wilkes} with $p=2$. That is, there exists a finitely generated
group $\Gamma$ that is inaccessible and residually $2$-finite. Since $\Gamma$ is finitely generated, it is countably infinite. 
By \Cref{lem:residually-p-tree}, $\Gamma$ acts faithfully on the rooted binary tree
$T_2$, hence
$\Gamma\leq \operatorname{Aut}(T_2)$. By \Cref{lem:aut-binary-tree} and \Cref{lem:subgroup}, $utww(\Gamma)\leq 2$.  


\end{proof}

\bibliographystyle{plainurlnat}
\bibliography{ch.bib}

\end{document}